\documentclass[twoside]{amsart}

\usepackage[latin9]{inputenc}
\usepackage{amsthm}
\usepackage{setspace}
\usepackage{amssymb}
\numberwithin{equation}{section} 
\numberwithin{figure}{section} 
\theoremstyle{plain}
\theoremstyle{plain}
\newtheorem{thm}{Theorem}
  \theoremstyle{plain}
  \newtheorem{prop}[thm]{Proposition}
  \theoremstyle{definition}
  \newtheorem{defn}[thm]{Definition}
  \theoremstyle{plain}
  \newtheorem{cor}[thm]{Corollary}
  \theoremstyle{plain}
  \newtheorem{lem}[thm]{Lemma}
  \theoremstyle{definition}
  \newtheorem*{problem*}{Problem}

\DeclareMathOperator{\homeo}{homeo}
\newcommand{\sH}{\mathcal{H}}

\begin{document}

\title{Rohlin properties for $\mathbb{Z}^{d}$ actions on the Cantor set}

\author{Michael Hochman}
\begin{abstract}
We study the space $\mathcal{H}(d)$ of continuous $\mathbb{Z}^{d}$-actions
on the Cantor set, particularly questions on the existence and nature
of actions whose isomorphism class is dense (Rohlin's property). Kechris
and Rosendal showed that for $d=1$ there is an action on the Cantor
set whose isomorphism class is residual; we prove in contrast that
for $d\geq2$ every isomorphism class in $\mathcal{H}(d)$ is meager.
On the other hand, while generically an action has dense isomorphism
class and the effective actions are dense, no effective action has
dense isomorphism class; thus conjugation on the space of actions
is topologically transitive but one cannot construct a transitive
point. Finally, we show that in the space of transitive and minimal
actions the effective actions are nowhere dense, and in particular
there are minimal actions that are not approximable by minimal SFTs.
\end{abstract}

\curraddr{Fine Hall, Washington Road, Princeton University, Princeton, NJ 08544 }

\email{hochman@math.princeton.edu}

\maketitle

\section{\label{sec:Introduction}Introduction}

Dynamical systems theory studies the asymptotic behavior of automorphisms
of some suitable structure, e.g. a measure space, topological space
or manifold, and more generally group actions by automorphisms. The
space of all such actions often carries a natural topology, and one
is led to questions about the distribution of isomorphism types in
the space of all actions, and the manner in which certain actions
can, or cannot, approximate others. In this way one hopes to achieve
some understanding of the relation between different types of dynamics.
Such a set-up is also suitable for studying rigidity phenomena, i.e.
the transmission of dynamical behaviors from certain actions to actions
in some neighborhood of it. For example see \cite{Halmos44,Ageev05,AlpernPrasad02,Smale91}. 

In ergodic theory this point of view is classical and goes back to
the work of Rohlin and Halmos \cite{Halmos44}, who studied the group
of automorphisms of a Lebesgue space equipped with the so-called coarse
topology. Of interest to us will be Rohlin's theorem that any aperiodic
automorphism has a dense isomorphism class. One should interpret this
as the assertion that, when viewed at any finite resolution, one cannot
distinguish any aperiodic isomorphism type from any other; all dynamical
types are {}``mixed together'' rather well. On the other hand, by
a theorem of del Junco there is a residual set of automorphisms disjoint
from any fixed ergodic automorphism, and hence no isomorphism class
can be residual; it follows from the general theory of Polish group
actions that every aperiodic isomorphism type, while dense, is meager.
These results hold for actions of more general groups, including $\mathbb{Z}^{d}$
actions (Rohlin's theorem at least holds for all discrete amenable
groups).

A natural analogue in the topological category is the space of continuous
actions on the Cantor set, which has received renewed attention recently.
We denote the Cantor space by $K$, and let \[
\mathcal{H}=\homeo(K)\]
be the Polish group of homeomorphisms of $K$ with the topology of
uniform convergence. Each $\varphi\in\mathcal{H}$ gives rise to a
$\mathbb{Z}$-action on $K$, which is the dynamical system associated
to it. Glasner and Weiss \cite{GW-TRproperty} showed that, as in
the ergodic-theory setting, there exist actions $\varphi\in\mathcal{H}$
whose isomorphism class is dense in $\mathcal{H}$ (although it is
not true that this is so for every aperiodic $\varphi$, as in Rohlin's
theorem). More recently this result has been subsumed by a remarkable
theorem of Kechris and Rosendal \cite{KR04}, who proved that there
is actually a single isomorphism class that is residual. Thus, generically,
there is only one $\mathbb{Z}$-actions on $K$. This action was later
described explicitly by Akin, Glasner and Weiss, and it turns out
to be rather degenerate, for example, it is not transitive; but in
the Polish space of transitive actions there is also a generic action
\cite{H08c}. 

The aim of the present paper is to begin the study of the space of
$\mathbb{Z}^{d}$-actions on $K$, which are of interest both in themselves
and as the topological systems underlying a large number of lattice
models in statistical mechanics and probability. We denote this space
by $\mathcal{H}(d)$; formally, it is defined by \[
\mathcal{H}(d)=\hom(\mathbb{Z}^{d},\mathcal{H})\]
with the Polish topology it inherits as a closed subset of the countable
product $\mathcal{H}^{\mathbb{Z}^{d}}$. The group $\mathcal{H}$
acts on $\mathcal{H}(d)$ by conjugation: if $\varphi\in\mathcal{H}(d)$
is an action $\{\varphi^{u}\}_{u\in\mathbb{Z}^{d}}$, and $\pi:K\rightarrow K$
is a homeomorphism, then the conjugation of $\varphi$ by $\pi$ is
the isomorphic action $\{\pi\varphi^{u}\pi^{-1}\}_{u\in\mathbb{Z}^{d}}$.
We denote the conjugacy class of $\varphi$ by $[\varphi]$; this
is by definition the set of actions in $\mathcal{H}(d)$ that are
isomorphic to $\varphi$. 

The group $\mathbb{Z}^{d}$ is said to have the \emph{weak topological
Rohlin property }(WTRP) if the there is an action in $\mathcal{H}(d)$
with dense conjugacy class; it has the \emph{strong topological Rohlin
property }(STRP) if there is a conjugacy class that is a dense $G_{\delta}$
(this is equivalent to there existing a conjugacy class containing
a dense $G_{\delta}$). This terminology has evolved recently in connection
with questions about the largeness of conjugacy classes in topological
groups and, more generally, largeness of conjugacy classes of the
actions of a fixed group; the archetypal example being Rohlin's result
on the group of measure-preserving automorphisms. See \cite{GW08}
for a recent survey and extensive bibliography.

As we have seen, $\mathbb{Z}$ has the strong (and therefore also
the weak) topological Rohlin property. The mechanism behind this is
a rather simple stability phenomenon whereby certain shifts of finite
type propagate their structure to nearby actions. To be precise,
\begin{thm}
Let $X$ be a $\mathbb{Z}^{d}$-shift of finite type and $\varphi\in\mathcal{H}(d)$
an action that factors into $X$. Then there is a neighborhood $U$
of $\varphi$ so that every action $\psi\in U$ factors into $X$.
\end{thm}
In particular if $X$ is minimal then all actions sufficiently close
to $X$ factor \emph{onto} $X$, and in dimension $1$ the same is
true if $X$ is a $0$-entropy SFT. This\emph{ }partly explains the
fact that the generic $\mathbb{Z}$-system of Kechris and Rosendal
as the countable product of all zero-entropy shifts of finite type
(the Akin-Glasner-Weiss description is somewhat different). The proof
of this also relies on some very special properties of zero-entropy
shifts of finite type in dimension $1$, particularly the fact that
their joinings decompose into countable many disjoint subsystems of
the same type. In contrast, in higher dimensions the same stability
phenomenon exists but the behavior of shifts of finite type is far
more complicated. 

Utilizing recent advances in our understanding of multidimensional
shifts of finite type, we are able to show that the case $d>1$ differs
from that of $d=1$:
\begin{thm}
\label{thm:STRP}For $d\geq2$, any action $\varphi\in\mathcal{H}(d)$
has a meager conjugacy class, i.e. $\mathbb{Z}^{d}$ does not have
the strong topological Rohlin property.
\end{thm}
It is much easier to establish that $\mathcal{H}(d)$ has the weak
Rohlin property. In fact this is a simple consequence of separability
of $\mathcal{H}(d)$, and holds for the space of actions of any discrete
groups: 
\begin{prop}
\label{thm:WTRP}For $d\geq2$, $\mathcal{H}(d)$ has the weak topological
Rohlin property, i.e. there are actions with dense conjugacy class.
\end{prop}
However, there is an interesting twist. In $\mathcal{H}(1)$, there
are explicit constructions of systems with dense conjugacy class.
There are several ways to make precise the notion of an explicit constructions,
one of which is the following. Say that $\varphi\in\mathcal{H}(d)$
is \emph{effective }if there is an algorithmic procedure for deciding,
given a finite set $F\subseteq\mathbb{Z}^{d}$ and a family $\{C_{u}\}_{u\in F}$
of closed and open subsets of $K$, whether $\cap_{u\in F}T^{u}C_{u}=\emptyset$.
In fact, one can (and we shall) weaken this and demand only that emptiness
of this intersection can be semi-decided, in the sense that if it
is empty the algorithm must detect this and halt, but may otherwise
it need not return a decision (see section \ref{sub:Effectively-closed-sets}
for a discussion and some other notions of effectiveness). The Kechris-Rosendal
can be realized as an effective action in both the stronger and weaker
sense, and one can also construct explicitly other actions with dense
orbit as Glasner and Weiss did. 

In contrast,
\begin{thm}
\label{thm:WTRP-is-non-effective}For $d\geq2$ there are no effective
actions with dense conjugacy class.
\end{thm}
Stated another way, the conjugation action of $\mathcal{H}$ on $\mathcal{H}(d)$
is topologically transitive, and therefore there is a dense $G_{\delta}$
set of actions whose conjugacy class is dense; but it is formally
impossible to construct a transitive point. 

In spite of the above, note that the effective systems are dense in
$\mathcal{H}(d)$ (e.g. the SFTs are dense; see proposition \ref{pro:density-of-SFTs}
below). But one cannot use the implication separable$\implies$WTRP,
as in proposition \ref{thm:WTRP}, to get en effective transitive
point, because the separability condition is not effective: there
is no recursive dense sequence of effective actions.

Nevertheless, density of effective systems means that in a certain
sense the entire space $\mathcal{H}(d)$ is accessible to us. It turns
out that this is not the case for some other interesting spaces. Consider
for example the Polish space $\mathcal{M}(d)\subseteq\sH(d)$ of minimal
actions, i.e. actions in which every orbit is dense. Classically such
actions have been studied extensively as the analogue of ergodic actions,
and there is a rich theory of their structure for arbitrary acting
groups \cite{Auslander88}. In dimension 1, one can explicitly construct
families of minimal actions that are dense in $\mathcal{M}(1)$; indeed,
in \cite{H08c} we showed that the universal odometer, i.e. the unique
(up to isomorphism) minimal subsystem of the product of all finite
cycles, is generic there, and in particular has a dense conjugacy
class; and in other reasonable parametrization one can show that the
finite cycles are dense (these fail to be dense in $\mathcal{H}(1)$
only for the technical reason that their phase space is not the Cantor
set). On the other hand, the following theorem shows that in higher
dimensions the space of minimal actions is in a very strong sense
inaccessible to us, even in the approximation sense:
\begin{thm}
\label{thm:minimal-systems}For $d\geq2$, the systems conjugate to
minimal effective systems systems are nowhere dense in $\mathcal{M}(d)$. 
\end{thm}
A similar statement holds for transitive systems. Note that a system
may be conjugate to an effective system without being effective itself. 

Note that SFTs are effective, and it follows that there are minimal
actions which cannot be approximated by SFTs. This is somewhat unexpected
as well: in dimension $d\geq2$ SFTs display a wealth of dynamics,
including minimal dynamics, and this has lead to the impression that
they can represent quite general dynamics. Theorem \ref{thm:minimal-systems}
shows that this is far from the case.

All the results above have an analogue in the space of closed subsystems
of the shift space $Q^{\mathbb{Z}}$, where $Q$ is the Hilbert cube
(the topology is that of the Hausdorff metric). This model was studied
in \cite{H08c} and the methods there can be used to translate the
present results to that setting.

The rest of this paper is organized as follows. In the next section
we prove the theorems about the WTRP and prove theorem \ref{thm:minimal-systems}.
Section \ref{sec:STRP} is devoted to the STRP. In section \ref{sec:Two-problems}
we conclude with some open questions.

\section{\label{sec:WTRP}The weak topological Rohlin property}

In this section we prove theorem \ref{sec:WTRP} and \ref{thm:minimal-systems}.
We first develop some basic facts about $\mathcal{H}(d)$.

\subsection{\label{sub:Actions-and-subshifts}Actions versus  subshifts}

When discussing an action $\varphi$ on $K$ we shall abbreviate and
write $\varphi$ for the associated dynamical system $(K,\varphi)$.
When dealing with a subshift $X$ of a symbolic space $\{1,2,\ldots,k\},^{\mathbb{Z}^{d}}$
or $K^{\mathbb{Z}^{d}}$, we denote by $\sigma=\{\sigma^{u}\}_{u\in\mathbb{Z}^{d}}$
the shift action and denote the system $(X,\sigma)$ simply by $X$.
We refer the reader to \cite{WALTERS82} for basic definitions from
topological dynamics.

There is a close connection between the space of actions of $\mathbb{Z}^{d}$
on $K$ and the space of subsystems of a shift space, with the Hausdorff
metric, and we shall have occasion to work with both settings. This
connection was explored in \cite{H08c}. Our presentation focuses
on the space $\sH$ but appeals to the subshift model at some points
to make use of symbolic constructions and invariants.

\subsection{\label{sub:Topology-of-H-and projections}Topology of $\mathcal{H}(d)$
and projection into SFTs}

Let $e_{1},\ldots,e_{d}$ denote the standard generators of $\mathbb{Z}^{d}$.
For concreteness let us fix a complete metric $d$ on $K$, and for
$\varphi,\psi:K\rightarrow K$ let \[
d(\varphi,\psi)=\max_{x\in K}d(\varphi(x),\psi(x))+\max_{y\in K}d(\varphi(y),\psi(y))\]
This is a complete metric on $\mathcal{H}$, and one verifies that
the metric\[
d(\varphi,\psi)=\max_{i=1,\ldots,d}d(\varphi^{e_{i}},\psi^{e_{i}})\]
is a complete metric on $\mathcal{H}(d)$ compatible with the topology
defined in the Introduction.

Given a partition $\alpha=\{A_{1},\ldots,A_{n}\}$ of $K$ into clopen
sets and an action $\varphi\in\sH(d)$, we write $c_{\alpha}:K\rightarrow\{1,\ldots,n\}^{\mathbb{Z}^{d}}$
for the coding map that takes $x\in K$ to its $\alpha$-itinerary,
i.e. to the sequence $(x_{u})_{u\in\mathbb{Z}^{d}}\in\{1,\ldots,n\}^{\mathbb{Z}^{d}}$
with $x_{u}=i$ if and only if $\varphi^{u}x\in A_{i}$. We write
$c_{\alpha,\varphi}$ when we want to make explicit the dependence
on $\varphi$. We also write $\widehat{c}_{\alpha}(\varphi)$ for
the image of the map $c_{\alpha,\varphi}$, which is a subshift of
$\{1,\ldots,n\}^{\mathbb{Z}^{d}}$. Thus $\widehat{c}_{\alpha}$ maps
actions to subshifts, and $c_{\alpha}=c_{\alpha,\varphi}$ is the
factor map from $(K,\varphi)$ to $(\widehat{c}_{\alpha}(\varphi),\sigma)$.

Recall that an SFT is a subshift $X\subseteq\Sigma^{\mathbb{Z}^{d}}$
($\Sigma$ finite) defined by a finite set of finite patterns $a_{1},\ldots,a_{n}$
and is the set of configurations $x\in\Sigma^{\mathbb{Z}^{d}}$ that
do not contain occurrences of any of the $a_{i}$. 

Although the SFT condition appears syntactic, it is an isomorphism
invariant. That is, if two subshifts are isomorphic and one is an
SFT, so is the other (defined by some other set of patterns). We shall
thus also refer to actions $\varphi\in\sH(d)$ as SFTs if they are
isomorphic to SFTs.
\begin{prop}
\label{pro:stability-of-SFTs}Suppose $\alpha=\{A_{1},\ldots,A_{n}\}$
is a clopen partition of $K$ and $\varphi\in\sH(d)$ is mapped via
$\widehat{c}_{\alpha}$ into a shift of finite type $X\subseteq\{1,\ldots,n\}^{\mathbb{Z}^{d}}$
(i.e. $\widehat{c}_{\alpha}(\varphi)\subseteq X$). Then there is
a neighborhood of $\varphi$ in $\sH(d)$ whose members are mapped
via $\widehat{c}_{\alpha}$ to subsystems of $X$.\end{prop}
\begin{proof}
Suppose $X\subseteq\{1,\ldots,n\}^{\mathbb{Z}^{d}}$ is specified
by disallowed patterns $b_{1},\ldots,b_{k}$ of diameter $<r$. Since
$A_{i}$ are clopen, it follows that whenever $\psi$ is an action
close enough to $\varphi$ then $\varphi^{u}(x),\psi^{u}(x)$ belong
to the same atom $A_{i}$ for every $u\in[-r,r]^{d}$ and $x\in K$.
Thus $c_{\alpha,\psi}(x)$ does not contain any of the $b_{i}$ and
so $c_{\alpha}(\psi)\subseteq X$.\end{proof}
\begin{prop}
\label{pro:density-of-SFTs}The shifts of finite type are dense in
$\sH(d)$.\end{prop}
\begin{proof}
Let $\varphi\in\sH$, $\varepsilon>0$, and choose a clopen partition
$\alpha=\{A_{1},\ldots,A_{n}\}$ of $K$ whose atoms are of diameter
$<\varepsilon$. Let $X\subseteq\{1,\ldots,n\}^{\mathbb{Z}^{d}}$
be the SFT specified by the condition that if for some $1\leq i,j\leq n$
and $1\leq k\leq d$ there is no $y\in K$ such that $y\in A_{i}$
and $\varphi^{e_{k}}y\in A_{j}$, then whenever $x\in X$ and $x(u)=i$
then $x(u+e_{k})\neq j$. Let $X'=X\times Y$ where $Y$ is the full
shift (this is only to ensure that $X'$ has no isolated points).
Choose a homeomorphism $\pi:X\times Y\rightarrow K$ mapping $[i]\times Y$
onto $A_{i}$ (here $[i]$ is the cylinder set $[i]\subseteq\{1,\ldots,n\}^{\mathbb{Z}^{d}}$),
and let $\psi=\pi\sigma\pi^{-1}$. One verifies that $d(\varphi,\psi)<\varepsilon$
and clearly $\psi$ is conjugate to the SFT $X'$.
\end{proof}

\subsection{\label{sub:Effectively-closed-sets}Effective dynamics}

There are a number of ways that one can define what it means for a
dynamical system is computable. We shall adopt a rather weak one;
at the end of this section we briefly discuss its relation to other
notions of computability.

A sequence $(a_{n})$ of integers is \emph{recursive} (R) if there
is an algorithm $A$ (formally a Turing machine) that, upon input
$n\in\mathbb{N}$, outputs $a_{n}$. A set of integers is \emph{recursively
enumerable} (RE) if it is the set of elements of some recursive sequence.
By identifying the integers with other sets we can speak of recursive
sequences of other elements. For example, since $\mathbb{N}\cong\mathbb{N}^{2}$
(and the bijection can be made effective), we can speak of recursive
sequences of pairs of integers; and in the same way of sequences of
finite sequences of integers. 

We shall assume from here on that the Cantor set is parametrized in
an explicit way. We shall use several such parametrization, representing
$K$ as $\{0,1\}^{\mathbb{N}}$, $\{1,2,\ldots,k\}^{\mathbb{Z}^{d}}$
and $(\{0,1\}^{\mathbb{N}})^{\mathbb{Z}^{d}}=K^{\mathbb{Z}^{d}}$.
All three may be identified by explicit homeomorphisms in such a way
that a family of cylinder sets in one is R or RE if and only if the
corresponding family of cylinder sets in the other parametrization
are also R or RE, respectively.

A subset $X\subseteq K$ is \emph{effective }if its complement is
the union of a recursive sequence of cylinder sets. Effective sets
are automatically closed and have been extensively studied in the
recursion theory literature, see e.g. \cite{Rogers67}.
\begin{defn}
A closed, shift-invariant subset $X\subseteq\{1,2,\ldots,k\}^{\mathbb{Z}^{d}}$
or $X\subseteq K^{\mathbb{Z}^{d}}$ is effective if it is an effectively
closed set.
\end{defn}
Note that this is not an isomorphism invariant. This is clear from
cardinality considerations, since there are countably many effective
subsets (each is defined by some algorithm), but there are uncountable
many ways to embed the full shift $\{0,1\}^{\mathbb{Z}^{d}}$, which
is effective, in $K^{\mathbb{Z}^{d}}$. However, effectiveness is
preserved under symbolic factors and is thus an invariant for symbolic
systems:
\begin{prop}
\label{pro:symbolic-factors-of-effective-systems}A symbolic factor
of an effective system is effective.
\end{prop}
For a proof see \cite[Proposition 3.3]{H07}. It is crucial that by
a symbolic factor we mean a factor that is a subsystem of $\{1,\ldots,k\}^{\mathbb{Z}^{d}}$
for some $k$. Although a subsystem $Y\subseteq\{x_{1},\ldots,x_{k})^{\mathbb{Z}^{d}}$
for points $x_{i}\in K$ is isomorphic to a symbolic system, such
a $Y$ may or may not be effective as a subsystem of $K^{\mathbb{Z}^{d}}$,
depending on the points $x_{i}$ (to see this it is enough to consider
fixed points of the shift action on $K^{\mathbb{Z}^{d}}$).

Since we are working in the space of actions, rather than the space
of subshifts, we introduce the following:
\begin{defn}
\label{def:effective-action}Let $\varphi\in\sH(d)$. A finite sequence
$\{(C_{i},u_{i})\}_{i=1}^{n}$, where $C_{i}$ is a cylinder set and
$u_{i}\in\mathbb{Z}^{d}$, is $\varphi$\emph{-disjoint }if\[
\bigcap_{u\in F}\varphi^{u}A_{u}=\emptyset\]
$\varphi$ is \emph{effective} if the set of $\varphi$-disjoint sequences
is RE, or in other words, if there is an algorithm that can recognize
a disjoint sequences in finite time (but may or may not identify non-disjoint
ones).
\end{defn}
This definition is related to effective subshifts as follows. Given
an action $\varphi\in\sH(d)$ and $x\in K$ let $\pi_{\varphi}(x)\in K^{\mathbb{Z}^{d}}$
be the point $(\pi_{\varphi}(x))_{u}=\varphi^{u}x$. Then $\pi_{\varphi}:K\rightarrow\pi(K)$
embeds $(K,\varphi)$ as the subsystem $(\pi_{\varphi}(K),\sigma)$
of $(K^{\mathbb{Z}},\sigma)$ (recall that $\sigma$ denotes the shift
action). One may verify that $\varphi$ is effective if and only if
$\pi_{\varphi}(X)$ is effective.%
\footnote{Note that if $Y\subseteq K^{\mathbb{Z}^{d}}$ is a closed and shift
invariant Cantor set which is effective, it can happen that $Y$ is
not of the form $Y=\pi_{\psi}(K)$ for any action $\psi$ on $K$.
For example, this is the case when $Y\subseteq\{x_{1},\ldots,x_{k}\}^{\mathbb{Z}^{d}}$
is a an infinite subshift for some fixed $x_{1},\ldots,x_{k}\in K$.%
}

Another way one might define effectiveness of an action $\varphi$
is to require that the maps $\varphi^{u}$ are computable. More precisely,
$\varphi$ is computable if there is an algorithm that, given $u\in\mathbb{Z}^{d}$,
an integer $n$ and a point $x\in K$, reads finitely many bits $x_{i}$
of $x$ and outputs $(\varphi^{u}x)_{n}$. This definition is similar
to that of Braverman and Cook \cite{BravermanCook06}, and implies
continuity of $\varphi^{u}$ and that the moduli of continuity are
computable. From this one can deduce that this notion is strictly
stronger than effectiveness in the sense of definition \ref{def:effective-action}.
Other definitions for effectiveness for sets, functions and dynamical
system have received some attention recently; see \cite{Grzegorczyk57,BrattkaPresser03,DelvenneKurkaBlondel06}.

\subsection{\label{sub:WTRP}Weak topological Rohlin Property}

Recall that a perfect space is one without isolated points. The non-effectiveness
part of theorem \ref{thm:WTRP} relies on the following:
\begin{thm}
\label{thm:no-universal-SFTs}No effective $\mathbb{Z}^{d}$-action
factors into every perfect SFT.
\end{thm}
The proof of this result is essentially identical to the proof given
in \cite{H08b}, where it was shown that if there were an effective
system that factors onto every SFT, then it could be used as part
of an algorithm that decides whether a given SFT is empty, and this
is undecidable by Berger's theorem \cite{B66,R71}. Two modifications
to the proof are needed to deduce the version above. First, an inspection
of the proof in \cite{H08b} shows that it does not use the fact that
the factor map is onto; thus the same proof works with the present
hypothesis that the map is into. Second, to prove the version above
we must show that, given the rules of an SFT which is either empty
or perfect, it is undecidable whether it is empty. But if we could
decide this, we could decide whether an arbitrary SFT were empty,
since an SFT $X$ is empty if and only if $X\times\{0,1\}^{\mathbb{Z}^{d}}$
is empty, and the latter is either empty or perfect. 
\begin{cor}
\label{pro:no-effective-transitive-action}If $\varphi\in\sH(d)$
has dense conjugacy class then it is not effective.\end{cor}
\begin{proof}
By proposition \ref{pro:density-of-SFTs} $\varphi$ would factor
into every $\psi\in\sH(d)$ that is conjugate to an SFT; hence it
would factor into every perfect SFT, and the conclusion follows from
theorem \ref{thm:no-universal-SFTs}.
\end{proof}
To complete the proof of theorem \ref{thm:WTRP} it remains to establish
that there is a dense conjugacy class in $\sH(d)$. The argument is
similar to that given in \cite{AGW06} for the measure-preserving
category, and works for any countable group. Since $\mathcal{H}(d)$
is separable, we may choose a dense sequence $\varphi_{1},\varphi_{2},\ldots$
in it, and let $\psi=\times_{i}\varphi_{i}$ be the product action
on $K^{\aleph_{0}}$, i.e.. \[
\psi(x_{1},x_{2}\ldots)=(\varphi_{1}(x_{1}),\varphi_{2}(x_{2}),\ldots)\]
It suffices to show that the closure of $[\psi]$ contains all the
$\varphi_{i}$. We work with the parametrization $K=\{0,1\}^{\mathbb{N}}$.
Fix an integer $n\in\mathbb{N}$; by uniform continuity of $\varphi_{1}^{e_{1}},\ldots,\varphi_{i}^{e_{d}}$
there is a $k(n)$ so that each $j=1,\ldots,d$ the first $n$ coordinates
of $\varphi_{1}^{e_{j}}(x)$ depend only on the first $k(n)$ coordinates
of $x$ for. Choose a partition $I_{1},I_{2},\ldots$ of $\mathbb{N}$
into infinite sets with $\{1,\ldots,k(n)\}\subseteq I_{1}$. Let $\pi_{m}:\mathbb{N}\rightarrow I_{m}$
be order-preserving isomorphisms and let $\tau=\tau_{n}$ be the action
on $K^{\aleph_{0}}$ which acts on $K^{I_{m}}$ like $\pi_{m}\varphi_{m}\pi_{m}^{-1}$.
The action $\tau$ is conjugate to $\psi$, and the first $n$ coordinates
of $\tau^{e_{j}}(x)$ and $\varphi_{1}^{e_{j}}(x)$ agree for all
$x\in K$. This proves the claim and completes the proof of theorem
\ref{thm:WTRP}.

\section{\label{sec:STRP}The strong Rohlin property}

In this section we prove theorem \ref{thm:STRP}. The key fact that
we use is that if there were a generic system $\theta\in\mathcal{H}(2)$
, then it has countably many symbolic factors $X_{1},X_{2},\ldots$,
since every symbolic factor of $\theta$ arises from a clopen partition
of $K$ and there are only countable many of these. We shall construct
an SFT $Y$ and associated action $\varphi$ so that, in a neighborhood
$U\subseteq\sH(d)$ of $\varphi$, a generic $\psi\in U$ has a symbolic
factor distinct from the $X_{i}$'s, and hence is not conjugate to
$\theta$. This factor will be the projection of $\psi$ into $Y$.

The main property we want of $Y$ is that its subsystems can be easily
perturbed. This will be accomplished by making the space of subsystems
if $Y$ be very rich. More precisely, there will be a sofic factor
$Z$ of $Y$ whose subsystems are not isolated, and furthermore if
$X\subseteq Z$ is effective then $X$ also has no isolated subsystems.
This is what will allow us to perturb the projections of actions into
$Y$. The control over subsystems will be achieved using recursive
methods.

\subsection{\label{sub:Medvedev-degrees}Medvedev degrees and dynamics }

Given $X\subseteq\{0,1\}^{\mathbb{N}}$, a function $f:X\rightarrow\{0,1\}^{\mathbb{N}}$
is \emph{computable }if there is an algorithm $A$ such that, when
given as input a point $x\in X$ (technically, $x$ is an oracle for
the computation) and an integer $k$, outputs the first $k$ coordinates
of $f(x)$. Note that $x$ is an infinite sequence of $0$ and $1$'s,
but the algorithm will perform finitely many operations before halting
so it will only read a finite number of these bits. Which bits it
chooses to read will depend on the bits it has already read and on
$k$. Thus if $x'$ differs from $x$ on coordinates which were not
used then running the algorithm on $x',k$ will give the same result
as $x,k$. It follows easily that a computable function is continuous
in the induced topology. 

An effective subset $Y\subseteq\{0,1\}^{\mathbb{N}}$ is reducible
to an effective subset $X\subseteq\{0,1\}^{\mathbb{N}}$ if there
is a computable function $f:X\rightarrow Y$ (not necessarily onto).
We denote this relation by $X\succ Y$. One should interpret this
as follows: suppose we want to show that $Y$ is not empty by producing
in some manner a point $y\in Y$. If $X\succ Y$ and if we can produce
a point $x\in X$ then we can, by applying the computable function
$f$, obtain the point $y=f(x)$. Thus $X$ is at least as complicated
as $Y$, in the sense that demonstrating that $X\neq\emptyset$ is
at least as hard as demonstrating that $Y$ $\neq\emptyset$. Notice
that if $X\subseteq Y$ then $X\succ Y$ (the identity map is computable),
and that if $y\in Y$ is computable as a function $y:\mathbb{N}\rightarrow\{0,1\}$
(i.e. if there is an algorithm that given $k$ computes the $k$-th
coordinate of $y$) then $X\succ Y$ for all $X$, because there is
a computable function $X\rightarrow\{y\}$, i.e. the map that doesn't
use the input $x$ at all and simply computes the components of $y$. 

We say that $X,Y$ are Medvedev equivalent if $X\succ Y$ and $Y\succ X$.
The equivalence class of $X$ is denoted $m(X)$ and called is Medvedev
degree of $X$. By the above, there is a minimal Medvedev degree consisting
of all effective sets containing computable points. There is also
a maximal element. There are infinitely many Medvedev degrees, and
they form a distributive lattice. Overall, the structure of this lattice
is still rather mysterious, although the theory of Medvedev degrees
is classical in recursion theory; see \cite{Rogers67}. 

Medvedev degrees were introduced into the study of SFTs by S. Simpson
\cite{S07}, who observed that, since a factor map between SFTs is
given by a sliding block code, the factoring relation $X\rightarrow Y$
between SFTs implies $m(X)\succ m(Y)$. This is true more generally
for effective symbolic systems and leads to the question, which is
still far from understood, of the relation between the Medvedev degree
of an effective system and its dynamics. One such connection is the
following, which will be central to our argument: 
\begin{prop}
\label{pro:no-complex-minimal-actions}If an effective subshift is
minimal then it has minimal Medvedev degree.
\end{prop}
The proof follows from \cite{H07}, proposition 9.4, where it was
shown for SFTs. 

We can now prove theorem \ref{thm:minimal-systems}. Suppose $Y$
is an SFT with non-minimal Medvedev degree, and let $Y_{0}\subseteq Y$
be a minimal subsystem. Let $\varphi$ be an action conjugate to $(Y_{0},\sigma)$.
Thus there is an open neighborhood $U\subseteq\sH(d)$ of $\psi$
so that every $\psi\in U$ factors into $Y$. 

We claim that $U$ does not contain any minimal SFTs. Indeed, if $\psi\in U$
were an action conjugate to a minimal SFT then every symbolic factor
of $(K,\psi)$ is effective and has minimal Medvedev degree. But this
would imply that $Y$ contains an effective subshift with minimal
degree and so itself has minimal degree, contrary to assumption.

The same argument shows that the effective systems are nowhere dense
in the space of transitive actions.

\subsection{\label{sub:Construction-of-Y}Construction of the SFT $Y$}

We construct an SFT $Y$ factoring onto a sofic shift $Y\rightarrow Z$,
so that $Z$ is the union of its minimal subsystems and has nontrivial
Medvedev degree.

Let $\Omega\subseteq\{0,1\}^{N}$ be an effective, closed set of non-trivial
Medvedev degree. To each $\omega\in\Omega$ we assign the point $y_{\omega}\in\{0,1\}^{\mathbb{Z}}$
defined as follows. First choose an effective enumeration of the integers:
$n(1),n(2),\ldots$. Select the coordinates of $y_{\omega}$ that
form the arithmetic progression of period $2$ passing through $n(1)$,
and assign to them the symbol $\omega(1)$. Next, choose the arithmetic
progression of period $4$ passing through the first of the $n(i)$
that is not yet colored, and assign to these coordinates the symbol
$\omega(2)$. At the $k$-th step, color with the symbol $\omega(k)$
the coordinates belonging to the arithmetic progression that passes
through the first uncolored $n(i)$. Let $Z_{\omega}$ denote the
orbit closure of $z_{\omega}$ and set $Z'=\cup_{\omega\in\Omega}Z_{\omega}$.
The map $\omega\mapsto z_{\omega}$ is a computable function $\Omega\rightarrow Z'$,
and there is also a recursive function $Z'\to\Omega$. One may verify
$Z'$ is closed and is effective. Thus $Z',\Omega$ are Medvedev equivalent.

To obtain an SFT $Y$ from $Z'$, we rely on the construction in section
6 of \cite{HM07}:
\begin{thm}
\label{thm:subactions-of-SFTs} There is a $\mathbb{Z}^{2}$ sofic
shift $Z$ such that $(Z,\sigma^{e_{1}})\cong Z'$ and $\sigma^{e_{2}}$
acts as the identity on $Z$. 
\end{thm}
Let $Y$ be an SFT factoring onto the sofic shift $Z$ and let $\rho:Y\rightarrow Z$
be the factor map. It is easily verified that $Y$ has the required
properties. We mention that $m(Y)$ is non-trivial because $Y\succ Z$
and $m(Z)=M(Z')\neq0$.

\subsection{\label{sub:Subsystems-and-extensions-of-Y}Subsystems and extensions
of $Y$}

Recall that if $X$ is a metric space with metric $d$, then the Hausdorff
distance between compact subsets $A,B\subseteq X$ is defined by the
condition that $d(A,B)<\varepsilon$ if and only if for each $a\in A$
there is a $b\in B$ with $d(a,b)<\varepsilon$ and the same with
the roles of $A,B$ reversed. The topology induced by the Hausdorff
metric is independent of the metric we began with, is compact when
$X$ is, and is totally disconnected if $X$ is.

Note that if $(X,\varphi)$ is a dynamical system then the space of
subsystems is closed in the Hausdorff metric. The following is elementary:
\begin{lem}
Let $W\subseteq\{1,\ldots,k\}^{\mathbb{Z}^{d}}$. If $W$ is an SFT
then the subshifts of $W$ which are SFTs are dense among the subsystems
of $W$; and similarly if $W$ is effective then its effective subsystems
are dense.\end{lem}
\begin{proof}
We prove the SFT case, the effective case being similar. Suppose $W$
is defined by disallowed patterns $\overline{b}=b_{1},\ldots,b_{m}$.
Fix a subsystems $X\subseteq W$, which we must show is an accumulation
point of SFTs. $X$ is defined by an infinite sequence of disallowed
patterns $b_{1},\ldots,b_{m},b_{m+1},b_{m+2},\ldots$ extending the
sequence $\overline{b}$. Let $X_{n}$ be the SFT defined by excluding
the patterns $b_{1},\ldots,b_{n}$; for $n\geq m$ we have $X_{n}\subseteq W$
and $\cap X_{n}=X$. It follows that $d(X,X_{n})\rightarrow0$, as
desired.
\end{proof}
Let $\rho:Y\rightarrow Z$ be the factor and systems constructed in
the previous section. 
\begin{lem}
\label{lem:minimal-not-isolated}If $X\subseteq Z$ is effective,
then in the space of subsystems of $X$ no minimal subsystem is isolated
in the Hausdorff metric.\end{lem}
\begin{proof}
Suppose $X'\subseteq X$ were an isolated minimal system. By the previous
lemma the effective subsystems of $X$ are dense, so $X'$ is an effective
minimal system and therefore has degree 0 by proposition \ref{pro:no-complex-minimal-actions},
implying the same for $X$ and therefore for $Z$, a contradiction.\end{proof}
\begin{lem}
\label{lem:minimals-perfect}If $X\subseteq Z$ is effective then
a minimal subsystem of $X$ is isolated if and only if its distance
from every other minimal subsystem of $X$ is bounded away from $0$. \end{lem}
\begin{proof}
Clearly if $X_{0}\subseteq X$ is isolated then its distance from
every other subsystem, and in particular the minimal ones, is bounded
away from $0$. 

Conversely, suppose there are systems arbitrarily close to $X_{0}$;
we must show that there are minimal systems arbitrarily close to $X_{0}$.
Let $\varepsilon>0$ and $x_{0}\in X_{0}$, and choose a finite $F\subseteq\mathbb{Z}^{d}$
so that $\{\sigma^{u}x_{0}\}_{u\in F}$ is $\varepsilon$-dense in
$X_{0}$. Let $X_{1}$ be a system $\varepsilon$-close to $X_{0}$,
and close enough that there is a point $x_{1}\in X_{1}$ such that
$d(\sigma^{u}x_{0},\sigma^{u}x_{1})<\varepsilon$ for $u\in F$. The
orbit closure $X_{1}'$ of $x_{1}$ is minimal since all subsystems
of $Z$ are. The proof will be completed by showing that $X'_{1}$
is within distance $2\varepsilon$ of $X_{0}$. To see this, note
that if $x\in X_{0}$ then $d(x,\sigma^{u}x_{0})<\varepsilon$ for
some $u\in F$, hence $d(x,T^{u}x_{1})<2\varepsilon$; and on the
other hand if $x'\in X'_{1}$ then $x'\in X_{1}$, so, since $d(X_{0},X_{1})<\varepsilon$,
there is a $x\in X$ with $d(x,x')<\varepsilon$. This implies $d(X_{0},X'_{0})<2\varepsilon$,
as required.\end{proof}
\begin{prop}
\label{pro:perurbation-of-subsystems}Let $Y$ by the SFT cover of
$Z$ constructed above, $W$ an SFT and $\pi:W\to Y$ a shift-commuting
map into a subsystem of $Y$. Then for every $\varepsilon>0$ there
is an SFT $W_{0}\subseteq W$ such that $d(W_{0},W)<\varepsilon$
in the Hausdorff metric and $\pi(W_{0})\neq\pi(W)$. \end{prop}
\begin{proof}
Consider the diagram\[
\begin{array}{ccc}
 &  & W\\
 &  & \pi\downarrow\\
Y & \supseteq & X=\pi(W)\\
\rho\downarrow &  & \rho\downarrow\\
Z & \supseteq & X'=\rho(X)\end{array}\]
$X'$ is a sofic shift, so it is effective. Let $C_{1},\ldots,C_{n}$
be a partition of $W$ into cylinder sets of diameter $<\varepsilon$.
Since $Z$ is the disjoint union of its minimal subsystems so is $X'$.
Hence by lemma \ref{lem:minimals-perfect}, none of the minimal subsystems
of $X'$ is isolated, and since $X$ is totally disconnected so is
the space of minimal subsystems. We can therefore partition $X'$
into clopen, pairwise disjoint invariant subsystems $X'_{1},\ldots,X'_{n+1}$.
For each $C_{i}$ there is at least one $X'_{j}$ such that $C_{i}\cap(\pi\rho)^{-1}(X'_{j})\neq\emptyset$.
Thus without loss of generality, $(\pi\rho)^{-1}(X'_{i})\cap C_{i}\neq\emptyset$,
and so $W'_{0}=\cup_{i=1}^{n}\pi^{-1}(X_{i})$ satisfies the desired
properties except it is not an SFT. But the subsystems that are SFTs
are dense among the subsystems of $W$ by lemma ???; we may therefore
choose a system $W_{0}$ with the requisite properties.
\end{proof}
It remains to translate this approximation lemma to the space $\sH(d)$.
\begin{cor}
\label{cor:perturbations-of-subsystems}Let $Y,W$ and $\pi:W\rightarrow Y$
be as in the previous lemma, and let $\varphi\in\sH(d)$ be conjugate
to $W$ by coding with respect to a partition $\alpha=\{A_{1},\ldots,A_{n}\}$
of $K$. Then for every $\varepsilon>0$ there is a SFT $\psi$ with
$d(\varphi,\psi)<\varepsilon$, and $\psi$ factors via $c_{\alpha,\psi}$
to an SFT $W_{0}\subseteq W$ with $\pi(W_{0})\neq\pi(W)$.\end{cor}
\begin{proof}
Since $\alpha$ generates for $\varphi$, there is an $r$ so that
the atoms of $\beta=\vee_{\left\Vert u\right\Vert <r}\varphi^{u}\alpha$
are of diameter $<\eta$ for a parameter $\eta$ we shall specify
later. By the previous lemma, we may choose a subshift $W_{0}\subseteq W$
so that $K_{0}=c_{\alpha,\varphi}^{-1}(W_{0})$ intersects each atom
of $\beta$. We define $\psi_{0}=\varphi|_{K_{0}}:K_{0}\rightarrow K_{0}$;
notice that $(K_{0},\psi_{0})\cong(W_{0},\sigma)$.

Since $W_{0}$ is an SFT it is effective, and since $W_{0}\subseteq W$
and $W$ has nontrivial degree, $W_{0}$ has no isolated points. Therefore
$K_{0}\cap B$ is topologically a Cantor set for each atom $B\in\beta$
and we may choose a homeomorphism $\rho:K\rightarrow K_{0}$ satisfying
$\rho(B)=K_{0}\cap B$ for $B\in\beta$. It follows that $d(x,\rho(x))<\eta$
for $x\in K$, so if $\eta$ is small enough, the action $\psi=\rho^{-1}\psi_{0}\rho$
will satisfy $d(\varphi,\psi)<\varepsilon$. Finally, the $\beta$-itineraries
of a point $x\in K$ are the same for the actions $\psi$ and $\psi_{0}$
since $\rho(B)=K_{0}\cap B$. Thus $c_{\beta,\psi}$ is a factor map
$(K,\psi)\rightarrow W_{0}$, and the lemma follows.
\end{proof}

\subsection{\label{sub:STRP}The strong topological Rohlin Property}

We now have all the parts we need to prove our main theorem.
\begin{thm}
\label{thm:main}For $d\geq2$ every isomorphism class in $\sH(d)$
is meager. \end{thm}
\begin{proof}
Fix $\theta\in\sH(d)$ and let $\varphi\in sH(d)$ be conjugate to
the SFT $Y\subseteq\{1,\ldots,k\}^{\mathbb{Z}^{d}}$ constructed above,
via a partition $\alpha$. Using proposition \ref{pro:stability-of-SFTs}
choose a neighborhood $U$ of $\varphi$ such that $\widehat{c}_{\alpha}(\psi)\subseteq Y$
for $\psi\in U$. We shall show that there is a residual subset $V\subseteq U$
of systems which are not isomorphic to $\theta$. 

Let $Y_{1},Y_{2},\ldots$ be an enumeration of all the subsystems
of $Y$ that are factors of $\theta$. It suffices to show that for
every $i=1,2,3\ldots$ there is a dense open set $V_{i}\subseteq U$
consisting of actions $\psi$ with $\widehat{c}_{\alpha}(\psi)\neq Y_{i}$,
for then $V=\cap V_{i}$ is a dense $G_{\delta}$ and if $\psi\in V$
then $\widehat{c}_{\alpha}(\psi)\neq Y_{i}$ for all $i$, implying
that $\psi\not\cong\varphi$.

Fix $i$ and let $\psi\in U$ be an SFT. If $\widehat{c}_{\alpha}(\psi)\neq Y_{i}$
then clearly any action $\psi'$ sufficiently close to $\psi$ will
also have $c_{\alpha}(\psi')\neq Y_{i}$. Thus we must show that the
SFTs $\psi$ with this property are dense in $U$. We already know
that the SFTs are dense, so let $\psi\in U$ be an SFT and suppose
$\widehat{c}_{\alpha}(\psi)=Y_{i}$. Then for every $\varepsilon>0$
we can apply corollary \ref{cor:perturbations-of-subsystems} to get
an SFT action $\psi'$ withing $\varepsilon$ of $\psi$, so that
$\widehat{c}_{\alpha}(\psi')\subseteq\widehat{c}_{\alpha}(\psi)$
and $\widehat{c}_{\alpha}(\psi')\neq\widehat{c}_{\alpha}(\psi)$. 

To conclude the proof we use the general fact that any orbit of a
Polish group acting transitively on a Polish space is either meager
or co-meager \cite{GW08}. So far we have shown that $[\theta]$ is
not residual, because it is not residual in the open set $U$; so
$[\theta]$ is meager. 
\end{proof}

\section{\label{sec:Two-problems}Two problems}

The picture emerging from these results is that the space of $\mathbb{Z}^{d}$
actions on $K$ is mostly inaccessible to us. The closure of the space
of effective systems may be better behaved. Here are a couple of questions
about this space.

Recall that an action $\varphi\in\sH(d)$ is strongly irreducible
if there is an $R>0$ such that, for every pair of open sets $\emptyset\neq A,B\subseteq K$,
we have $\varphi^{u}A\cap B\neq\emptyset$ for every $u\in\mathbb{Z}^{d}$
with $\left\Vert u\right\Vert \geq R$. The class of SFTs with this
property has been widely studied in thermodynamics as the class with
the best hope of developing something of a thermodynamic formalism,
and in symbolic dynamics as a fairly manageable class where embedding
and factoring relations may be well behaved (note that the factor
of a strongly irreducible system is itself strongly irreducible). 
\begin{problem*}
Can every strongly irreducible action be approximated by a strongly
irreducible SFT?
\end{problem*}
In dimension 1 the answer is affirmative. Note that strongly irreducible
SFTs, like minimal SFTs, have Medvedev degree 0 \cite[Corollary 3.5]{HM07}.
Thus a negative answer would follow if we could construct an SFT of
non-trivial degree having some strongly irreducible subsystem (which
of course will not be effective). 

With regard to the space of minimal systems, we have shown that the
(relative) closure of the effective systems, and thus of the minimal
SFTs, has empty (relative) interior. It is still open if these closures
are the same. In other words,
\begin{problem*}
Can every minimal effective system be approximated by a minimal SFT?
\end{problem*}
\bibliographystyle{plain}
\bibliography{bib,recursion-theory}

\end{document}